\documentclass[12pt,oneside]{amsart}
\usepackage[margin=1in]{geometry}% See geometry.pdf to learn the layout options. There are lots.

\usepackage{graphicx}
\usepackage{amssymb}
\usepackage{tikz}
\usepackage{latexsym,epsfig,enumerate,multicol,wasysym}
\usepackage{amsmath,amsthm,amscd,amssymb}
\usepackage{latexsym}
\usepackage[colorlinks,citecolor=red,pagebackref,hypertexnames=false]{hyperref}
\hypersetup{%
  colorlinks = true,
  citecolor=red,
  linkcolor  = red
}

\usepackage{ulem}
\usepackage{soul}
\numberwithin{equation}{section}
\theoremstyle{plain}
\newtheorem{theorem}{Theorem}[section]
\newtheorem{lemma}[theorem]{Lemma}
\newtheorem{corollary}[theorem]{Corollary}
\newtheorem{proposition}[theorem]{Proposition}

\theoremstyle{definition}
\newtheorem{definition}[theorem]{Definition}

\theoremstyle{remark}
\newtheorem{remark}[theorem]{Remark}

\newtheorem{case[theorem]}{Case}

\date{\today}      
\author{K. Aldahleh, A. Iosevich, J. Iosevich, J. Jaimangal, A. Mayeli, S. Pack} 
\address{Department of Mathematics, University of Rochester, Rochester, NY}
\email{kaldahle@u.rochester.edu}
\address{Department of Mathematics, University of Rochester, Rochester, NY}
\email{iosevich@gmail.com}
\address{Department of Mathematics, Rochester Institute of Technology, Rochester, NY}
\email{joshuaiosevich@gmail.com}
\address{Department of Mathematics, CUNY Graduate Center, New York, NY}
\email{jonathanjaimangal1027@gmail.com}
\address{Department of Mathematics, CUNY Graduate Center, New York, NY}
\email{amayeli@gc.cuny.edu}
\address{Department of Mathematics, Penn State University, University Park, PA}
\email{skp6221@psu.edu}
\thanks{A.I. was supported in part by the National Science Foundation under grant no. 2154232. A.M. was supported in part by AMS-Simons Research Enhancement Grant, Simon Fellowship,  and the PSC-CUNY research grants.}

\begin{document}

% \title[Signal recovery]{Signal recovery and harmonic analysis}
%\title{On the interplay of sparse signal Recovery and Restriction Theorems on Finite Abelian Groups}
\title[Additive energy uncertainty principle]{Additive energy, uncertainty principle and signal recovery mechanisms}

\begin{abstract} Given a signal $f:G\to\mathbb{C}$, where $G$ is a finite abelian group, under what reasonable assumptions can we guarantee the exact recovery of $f$ from a proper subset of its Fourier coefficients? In 1989, Donoho and Stark established a result \cite{DS89} using the classical uncertainty principle, which states that $|\text{supp}(f)|\cdot|\text{supp}(\hat{f})|\geq |G|$ for any nonzero signal $f$. Another result, first proven by Santose and Symes \cite{SS86}, was based on the Logan phenomenon \cite{L65}. In particular, the result showcases how the $L^1$ and $L^2$ minimizing signals with matching Fourier frequencies often recovers the original signal.

The purpose of this paper is to relate these recovery mechanisms to additive energy, a combinatorial measure denoted and defined by
$$\Lambda(A)=\left| \left\{ (x_1, x_2, x_3, x_4) \in A^4 \mid x_1 + x_2 = x_3 + x_4 \right\} \right|,$$
where $A\subset\mathbb{Z}_N^d$. In the first part of this paper, we use combinatorial techniques to establish an improved variety of the uncertainty principle in terms of additive energy. In a similar fashion as the Donoho-Stark argument, we use this principle to establish an often stronger recovery condition. In the latter half of the paper, we invoke these combinatorial methods to demonstrate two $L^p$ minimizing recovery results.
\end{abstract}

\maketitle

\tableofcontents

\section{Introduction} \label{tssection} Let $f:{\mathbb Z}_N^d \to {\mathbb C}$ be a function on the $d$-dimensional module over $\mathbb{Z}_N:=\mathbb{Z}/N\mathbb{Z}$, where $N$ is an arbitrary positive integer $\ge 2$. %The discrete Fourier transform (DFT) is a very useful tool for time series analysis, often used to clean the time series data by removing the noise. See, for example, \cite{wang2017radar,yang2012stereophonic, tairan2016fourier}, and the references contained therein.
Given $f: {\mathbb Z}_N^d \to \mathbb{C}$, 
  the discrete Fourier transform  (DFT) is defined by $\widehat{f}:\mathbb{Z}_N^d\rightarrow\mathbb{C}$ such that for $m\in\mathbb{Z}_N^d$
$$ \widehat{f}(m)=N^{-d} \sum_{x \in {\mathbb Z}_N^d} \chi(-x \cdot m) f(x), \ m \in {\mathbb Z}_N^d,$$ 
where $\chi(t)=e^{\frac{2\pi i}{N}t}$ for $t\in\mathbb{Z}_N$.
The inverse discrete Fourier transform is given by 

$${f}(x)= \sum_{m \in {\mathbb Z}_N^d} \chi(x \cdot m) \widehat{f}(m), \ x \in {\mathbb Z}_N^d,$$ 
 
If $f$ is encoded via its Fourier transform, it is natural to ask whether $f$ can be recovered from its Fourier coefficients if part of the transmission is lost. In particular, suppose that the values of $\{\widehat{f}(m)\}_{m\in S}$ are {\it not observed} for $S$, a subset of  $\mathbb{Z}_N^d$.
By this, we mean that the discrete Fourier transform $\hat f(m)$ is only given for the elements $m\not\in S$. 
The problem of recovery of signals when some frequency components are missing was studied in the seminal paper by Matolcsi and Szucs (\cite{MS73}), and, independently, by Donoho and Stark (\cite{DS89}). In particular, Donoho and Stark proved the following result, which we state in arbitrary dimensions, as per discussion in \cite{IM24}. 

\begin{theorem}[\cite{DS89}]\label{DS} Let $f: {\mathbb Z}_N^d \to \Bbb C$ be a finite signal in ${\mathbb Z}_N^d$ with $N_t\in \Bbb N$ non-zero entries. Suppose that the set of unobserved frequencies $\{\hat f(m)\}_{m\in \mathbb Z_N}$ is of size $N_w\in \Bbb N$. Then the signal $f$ can be recovered uniquely from the observed frequencies if 
\begin{align} \label{sufficient-size-for-recovery} N_t \cdot N_w < \frac{N^d}{2}.
\end{align} \end{theorem}

To establish uniqueness,  Donoho and Stark used the classical Fourier uncertainty principle (see e.g. \cite{DS89}) which says that for a nonzero signal $f:\mathbb{Z}_N^d\to\mathbb{C}$,
$$ |\text{supp}(f)| \cdot |\text{supp}(\widehat{f})| \ge N^d.$$

Here, 
$|\text{supp}(f)|$ 
denotes the number of nonzero entries of the signal $f$. 
Indeed, suppose that $f$ cannot be recovered uniquely. Then there exists a signal $g: {\mathbb Z}_N^d\rightarrow\mathbb{C}$ such that 
$|\text{supp}(f)|=|\text{supp}(g)|$, $\widehat{f}=\widehat{g}$ away from the set of missing frequencies, and $f$ is not identically equal to $g$. Let $h=f-g$. Then 
$$ |\text{supp}(h)| \leq 2 |\text{supp}(f)|$$ and $\text{supp}(\widehat{h})$ is the subset of $S$, the set of missing frequencies. By the uncertainty principle, 
$$ 2|\text{supp}(f)| \cdot |S| \ge N^d.$$

Therefore, if we assume that  
\begin{equation} \label{eq: donohostarkrecoverycondition} |\text{supp}(f)| \cdot |S|<\frac{N^d}{2}, \end{equation} we must conclude that $h \equiv 0$, proving Theorem \ref{DS}. Donoho and Stark provide a recovery mechanism based on the least squares method, but this method is rather computationally expensive (exponential time). This method is described in Section \ref{section:minimization} below. 

While more efficient recovery mechanisms have been developed (see e.g. the $l^1$ minimization in \cite{DS89}), they are still rather elaborate. On the other hand, a very simple recovery mechanism was developed in \cite{IM24} for binary signals. Let $E(x)$ be the indicator function of a set $E \subset {\mathbb Z}^d_N$ and $\widehat{E}$ be the Fourier transform of $E$. Suppose that the frequencies $\widehat{E}(m)$ for  $m$ in  $S$ are unobserved. Then, by the Fourier inversion, 
\begin{equation} \label{eq:DRA} E(x)=\sum_{m\in \Bbb Z_N^d} \chi(x \cdot m) \widehat{E}(m)=\sum_{m \notin S} \chi(x \cdot m) \widehat{E}(m)+\sum_{m \in S} \chi(x \cdot m) \widehat{E}(m)=I(x)+II(x). \end{equation} 

Applying the triangle inequality and the definition of the Fourier transform, we see that 
$$ |II(x)| \leq N^{-d}|E||S|.$$

If this quantity is $<\frac{1}{2}$, we can take $|I(x)|$ and round it to $1$ or $0$. This recovers $E(x)$ exactly since it only takes on values $0$ or $1$. This recovery algorithm is known as the {\it Direct Recovery Algorithm}.

\begin{remark} \label{rmk:discretization} The same idea works if $f$ takes values in $\delta {\mathbb Z}$ for some $\delta>0$. In this case, we need to make sure that the modulus of $II(x)$ above is bounded by $\frac{\delta}{2}$ instead of $\frac{1}{2}$, which allows us to recover the original signal exactly by taking the modulus of $I(x)$ and rounding to the nearest value in $\delta {\mathbb Z}$. 
\end{remark} 

\vskip.125in 

There is a number of results in current literature that improve the recovery condition (\ref{eq: donohostarkrecoverycondition}). For example, the following approach was used by the second and fifth listed authors in \cite{IM24}. Their main tool was the following celebrated result due to Jean Bourgain (\cite{Bourgain89}). 

\begin{theorem} \label{bourgaintheorem} Let $\Psi=(\psi_1, \dots, \psi_n)$ denote a sequence of $n$ mutually orthogonal functions, with ${||\psi_i||}_{L^{\infty}(G)} \leq 1$. There exists an index subset $S$ of $\{1,2, \dots, n\}$, with  $|S|>n^{\frac{2}{q}}$,  such that 
$$ {\left|\left| \sum_{i \in S} a_i \psi_i \right|\right|}_{L^q(G)} \leq C(q) {\left( \sum_{i \in S} {|a_i|}^2 \right)}^{\frac{1}{2}}.$$ 

The constant $C(q)$ depends only on $q$ and the estimate above holds for a generic set of size $\lceil n^{\frac{2}{q}} \rceil$, where $\lceil x\rceil$ denotes the smallest integer greater than $x$. 
\end{theorem} 

\vskip.125in 

\begin{remark}\label{remarkgeneric}   
The notion of {\bf generic} in Theorem \ref{bourgaintheorem} means the following. Let $0<\delta<1$ and let ${\{\xi_j\}}_{1 \leq j \leq n}$ denote independent $0,1$ random variables of mean $\int \xi_j(\omega) d\omega=\delta$, $1 \leq j \leq n$. Choosing $\delta=n^{\frac{2}{q}-1}$ generates a random subset $S_{\omega}=\{1 \leq j \leq n: \xi_j(\omega)=1\}$ of $\{1,2, \dots n\}$ of expected size $\lceil n^{\frac{2}{q}} \rceil$. Theorem \ref{bourgaintheorem} holding for a {\bf generic} set $S$ means that the result holds for the set $S_{\omega}$ with probability $1-o_N(1)$. In simpler language, if we randomly choose a subset of $\{1,2, \dots, n\}$ by choosing each element with probability $p=n^{\frac{2}{q}-1}$, then Theorem \ref{bourgaintheorem} holds for such a set with probability close to $1$. 
\end{remark} 

\vskip.125in 

The following consequence of Theorem \ref{bourgaintheorem} is particularly relevant to our investigation. 

\begin{corollary} \label{bourgaindiscretecorollary} Given $f: {\mathbb Z}_N^d \to {\mathbb C}$,   for a generic subset $\Sigma$ of ${\mathbb Z}_N^d$ of size $\lceil N^{\frac{2d}{q}} \rceil$, where $q>2$, if $\widehat{f}$ is supported in $\Sigma$, we have 
\begin{equation} \label{ZNBourgain} {||f||}_{L^q(\mu)} \leq C(q) {||f||}_{L^2(\mu)}, \end{equation} where $C(q)$ depends only on $q$.  

%\footnote{
%\color{red} Could you remind me why this inequality is independent of $\Sigma$? Should not this be like $\|\sum_{m\in \Sigma} \hat f(m) \chi_m\|_{L^q} \leq C(q)\|\sum_{m\in \Sigma} \hat f(m) \chi_m\|_{L^2}$? -- Azita} \footnote{\color{purple} Bourgain describes his proof for a generic subset in his paper on section 5 here: \\
%\url{https://link.springer.com/article/10.1007/BF02392838} \\
%The corollary should follow from letting $\Psi$, in Theorem 1.3, denote a sequence of $\psi_m(x)=\chi(x\cdot m)$ for each $m\in\mathbb{Z}_N^d$ and $a_m=\hat{f}(m)$. The left hand side of Theorem 1.3 ends up being $||f||_{L_q(\mu)}$. The right hand side should just be $||f||_{L^2(\mu)}$ up to a constant by Plancherel's. -- Karam}

\vskip.125in 

Here and throughout, we define 
\begin{equation} \label{Lpnormalized} {||f||}_{L^p(\mu)}:={\left(\frac{1}{N^d} \sum_{x \in {\mathbb Z}_N^d} {|f(x)|}^p \right)}^{\frac{1}{p}}. \end{equation} 
\end{corollary} 

The authors in \cite{IM24} used Corollary \ref{bourgaindiscretecorollary} to prove that if $S$ is generic of size $\lceil N^{\frac{2d}{q}} \rceil$ in the sense described above, $\widehat{f}$ is supported in $S$, and $f$ is supported in $E \subset {\mathbb Z}_N^d$, then 
\begin{equation} \label{eq: iosevichmayeli} |E| \ge \frac{N^d}{{(C(q))}^{\frac{1}{\frac{1}{2}-\frac{1}{q}}}}. \end{equation} 

Indeed, the left-hand side of (\ref{ZNBourgain}) can be rewritten as 
$$ N^{-\frac{d}{q}} {|E|}^{\frac{1}{q}} {\left( \frac{1}{|E|} \sum_x {|f(x)|}^q \right)}^{\frac{1}{q}}, $$ while the right-hand side can be rewritten as 
$$ C(q) N^{-\frac{d}{2}} {|E|}^{\frac{1}{2}} {\left( \frac{1}{|E|} \sum_x {|f(x)|}^2 \right)}^{\frac{1}{2}}.$$

Combining terms and using the fact that 
$$ {\left( \frac{1}{|E|} \sum_x {|f(x)|}^2 \right)}^{\frac{1}{2}} \leq {\left( \frac{1}{|E|} \sum_x {|f(x)|}^q \right)}^{\frac{1}{q}}$$ by H\"older's inequality yields (\ref{eq: iosevichmayeli}). 

Using the Donoho-Stark signal argument described above, Iosevich and Mayeli proved that if $f: {\mathbb Z}_N^d \to {\mathbb C}$ is a signal supported in $E$, and the frequencies ${\{\widehat{f}(m)\}}_{m \in S}$ are unobserved, where $S$ is a generic subset of ${\mathbb Z}_N^d$ of size $\lceil N^{\frac{2}{q}} \rceil$, then $f$ can be recovered exactly and uniquely provided that 
$$ |E|<\frac{N^d}{2{(C(q))}^{\frac{1}{\frac{1}{2}-\frac{1}{q}}}}.$$

This result shows that much better recovery procedures are possible if the set $S$ of missing frequencies satisfies additional structural hypotheses (see also other results in \cite{IM24}). The purpose of this paper is to prove a user-friendly result of this type, where the structure of $S$ is measured in terms of an easily computable quantity called additive energy. 

\vskip.25in 

\subsection{Additive energy uncertainty principle} Additive energy is a combinatorial property with a strong relationship to the Fourier transform, as we will soon demonstrate.

\begin{definition}   [Additive Energy]
Let \(A \subset \mathbb{Z}_N^d\). The {\bf additive energy} of \(A\), denoted by \(\Lambda(A)\), is given by 
\[\Lambda(A) = \left| \left\{ (x_1, x_2, x_3, x_4) \in A^4 \mid x_1 + x_2 = x_3 + x_4 \right\} \right|.\] 
\end{definition}

% Notice that for any subset $A\subset \Bbb Z_N^d$, we have 
% $$2|A|^2-|A| \leq \Lambda(A) \leq |A|^3.$$

% % The lower bound is based on the simple observation that $x_1+x_2=x_3+x_4$ always holds if $x_1=x_3, x_2=x_4$, or if $x_1=x_4, x_2=x_3$. We subtract the size of $A$ to compensate for instances when $x_1=x_2$. The upper bound is obtained by fixing $x_1,x_2,x_3$ and expressing $x_4$ in terms of the other variables. The question then boils down to how many triplets, $(x_1,x_2,x_3)\in A^3$, are such that $x_1+x_2-x_3\in A$. The best case scenario is that this always holds true for any triplet, implying an upper bound on $\Lambda(A)$ by $|A|^3$.

We first demonstrate the technique we use to relate additive energy with recovery in the binary signal case. Let $E\subset\mathbb{Z}_N^d$ and we define the indicator function $E(x)$
by  $E(x)=1$ when $x\in E$ and $E(x)=0$ otherwise. Suppose $\{\widehat{E}(m)\}_{m\in S}$ is lost for $S\subset\mathbb{Z}_N^d$. As before, by inverse discrete Fourier transform,  we can write  
$$ E(x)=\sum_{m \notin S} \chi(x \cdot m) \widehat{E}(m)+\sum_{m \in S} \chi(x \cdot m) \widehat{E}(m)=I(x)+II(x).$$
Instead of applying the triangle inequality, we estimate $|II(x)|$ by H\"older's inequality for the conjugate pair $(4/3, 4)$.
\begin{align*}
\left|\sum_{m \in S} \chi(x \cdot m) \widehat{E}(m)\right|&\leq |S|^\frac{3}{4}\left(\sum_{m\in\mathbb{Z}_N^d}\left|\widehat{E}(m)\right|^4\right)^\frac{1}{4}=|S|^\frac{3}{4}\left(\sum_{m\in\mathbb{Z}_N^d}\left|\sum_{x\in E}\chi(-x\cdot m)E(x)\right|^4\right)^\frac{1}{4}\\ 
&=|S|^\frac{3}{4}N^{-d}\left(\underset{x_3,x_4\in E}{\sum_{x_1,x_2,}}\sum_{m\in\mathbb{Z}_N^d}\chi((x_3+x_4-x_1-x_2)\cdot m)E(x_1)E(x_2)\overline{E(x_3)E(x_4)}\right)^\frac{1}{4}\\
&=|S|^\frac{3}{4}N^{-\frac{3d}{4}}\left(\underset{x_i\in E}{\sum_{x_1+x_2=x_3+x_4;}}E(x_1)E(x_2)\overline{E(x_3)E(x_4)}\right)^\frac{1}{4}=|S|^\frac{3}{4}N^{-\frac{3d}{4}}\Lambda^\frac{1}{4}(E)
\end{align*}
To establish a recovery condition, we need the right-hand side to be less than $\frac{1}{2}$. Simplifying we arrive at the proceeding recovery condition.
\begin{proposition}
    If $E:\mathbb{Z}_N^d\to\mathbb{C}$ is the indicator function of $E\subset\mathbb{Z}_N^d$ and $\{\widehat{E}(m)\}_{m\in S}$ are lost for $S\subset\mathbb{Z}_N^d$, then $f$ is recoverable by the Direct Recovery Algorithm provided that
    $$|S|\Lambda^\frac{1}{3}(E)<\frac{N^d}{2^\frac{4}{3}}.$$
\end{proposition}
Although this result is intriguing and often stronger than the previous Direct Recovery result (see (\ref{eq:DRA}) above), there's certainly a lot more generalization one can procure along this line of reasoning. To incorporate the use of energy for recovery on any signal, we discovered two new varieties of the uncertainty principle. The following result is interesting in its own right.

\begin{theorem}[Additive Uncertainty Principle]\label{energyuncertaintyprincipletheorem}
Let  $f:\mathbb{Z}_N^d \rightarrow \mathbb{C}$ be a nonzero signal with support in $E$,  and let $\widehat{f}$ denote its Fourier transform with support in $\Sigma$. 
 Then for any  $\alpha\in[0,1]$,
\begin{itemize} 
\item[i)] \label{additiveenergyeq} $N^d \leq {(|E| \cdot \Lambda^{\frac{1}{3}}(\Sigma))}^{1-\alpha} \cdot {( \Lambda^{\frac{1}{3}}(E) \cdot |\Sigma|)}^{\alpha}.$
\item[ii)] \label{additiveenergyrestrictioneq} $N^d \leq {\left(|E| \max_{U \subset \Sigma} \frac{\Lambda(U)}{{|U|}^2}\right)}^{1-\alpha} \cdot {\left(|\Sigma| \max_{F \subset E} \frac{\Lambda(F)}{{|F|}^2} \right)}^{\alpha}$.
\end{itemize}
\end{theorem}

\vskip.125in 

\begin{remark} To prove part i), it is sufficient to establish the inequality $N^d \leq |E| \cdot \Lambda^{\frac{1}{3}}(\Sigma)$. The inequality $N^d \leq |\Sigma| \cdot \Lambda^{\frac{1}{3}}(E)$ follows by reversing the roles of $E$ and $\Sigma$, and the general case follows from these two by writing $N^d=N^{d(1-\alpha)} \cdot N^{d \alpha}$, $0 \leq \alpha \leq 1$. 
\end{remark}

Perhaps the key technical point in the proof of the first part of Theorem \ref{energyuncertaintyprincipletheorem} is the fact that if $S \subset {\mathbb Z}_N^d$, then 
$$ \Lambda(S)=N^{3d} \sum_{m \in {\mathbb Z}_N^d} {|\widehat{S}(m)|}^4.$$

To illustrate this point visually, let $S$ be a random subset of ${\mathbb Z}_{100}$ of size $10$. Then the plot of moduli of the Fourier coefficients of the indicator function of $S$ is given in Figure \ref{FTrandombinary}. It is quite apparent that aside from the $0$ coefficient, the rest of the moduli of the Fourier coefficient are quite small. On the other hand, suppose that $S$ is an arithmetic progression of length $10$ in ${\mathbb Z}_{100}$. The plot of the moduli of the Fourier coefficients is given in Figure \ref{FTprogressionbinary}. In this case, the modulus of the Fourier coefficients takes on several large values. 

\begin{figure}
\centering
\includegraphics[scale=.4]{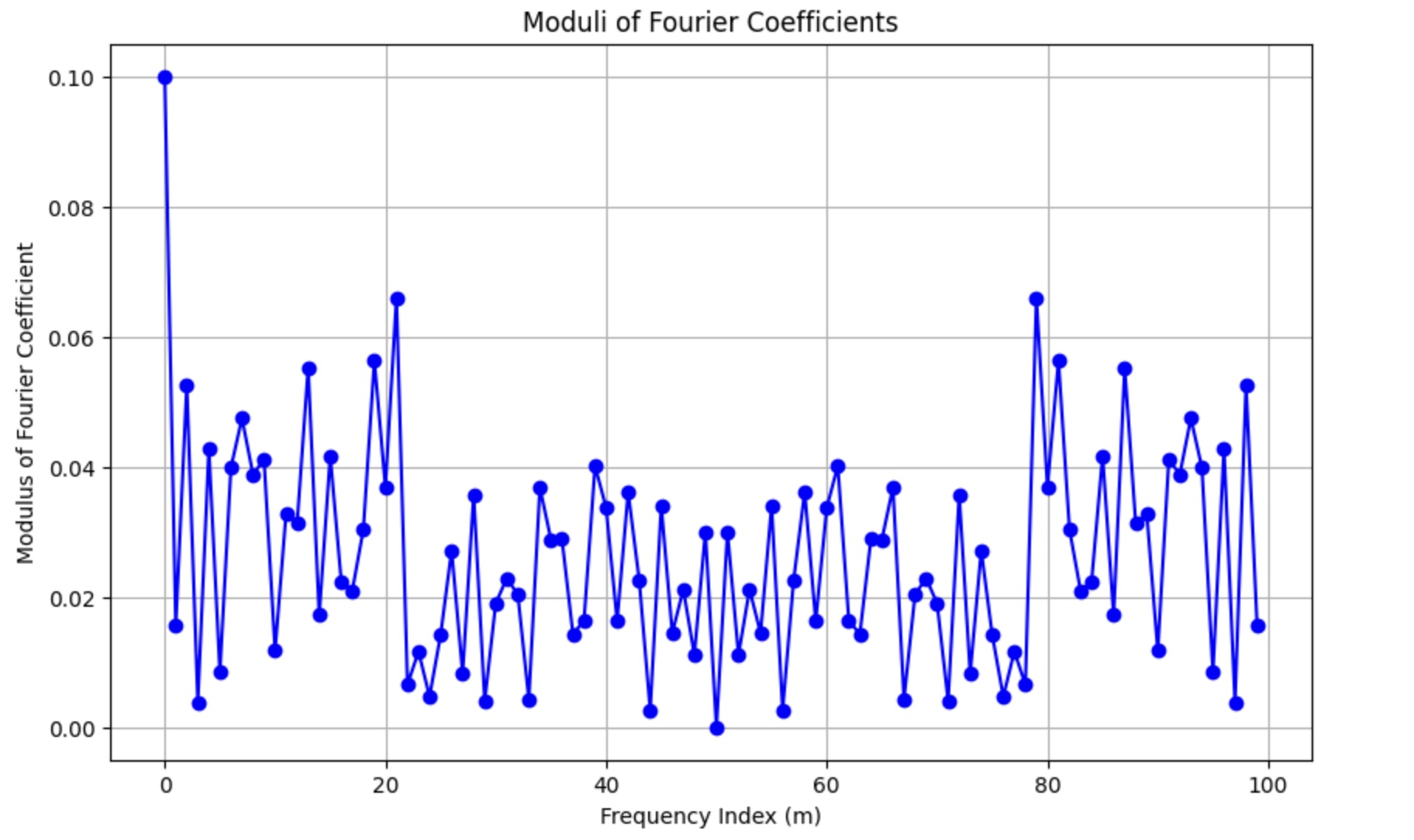}
 \caption{The discrete Fourier transform of a random set.}
 \label{FTrandombinary}

\end{figure}

\begin{figure}
\centering
\includegraphics[scale=.4]{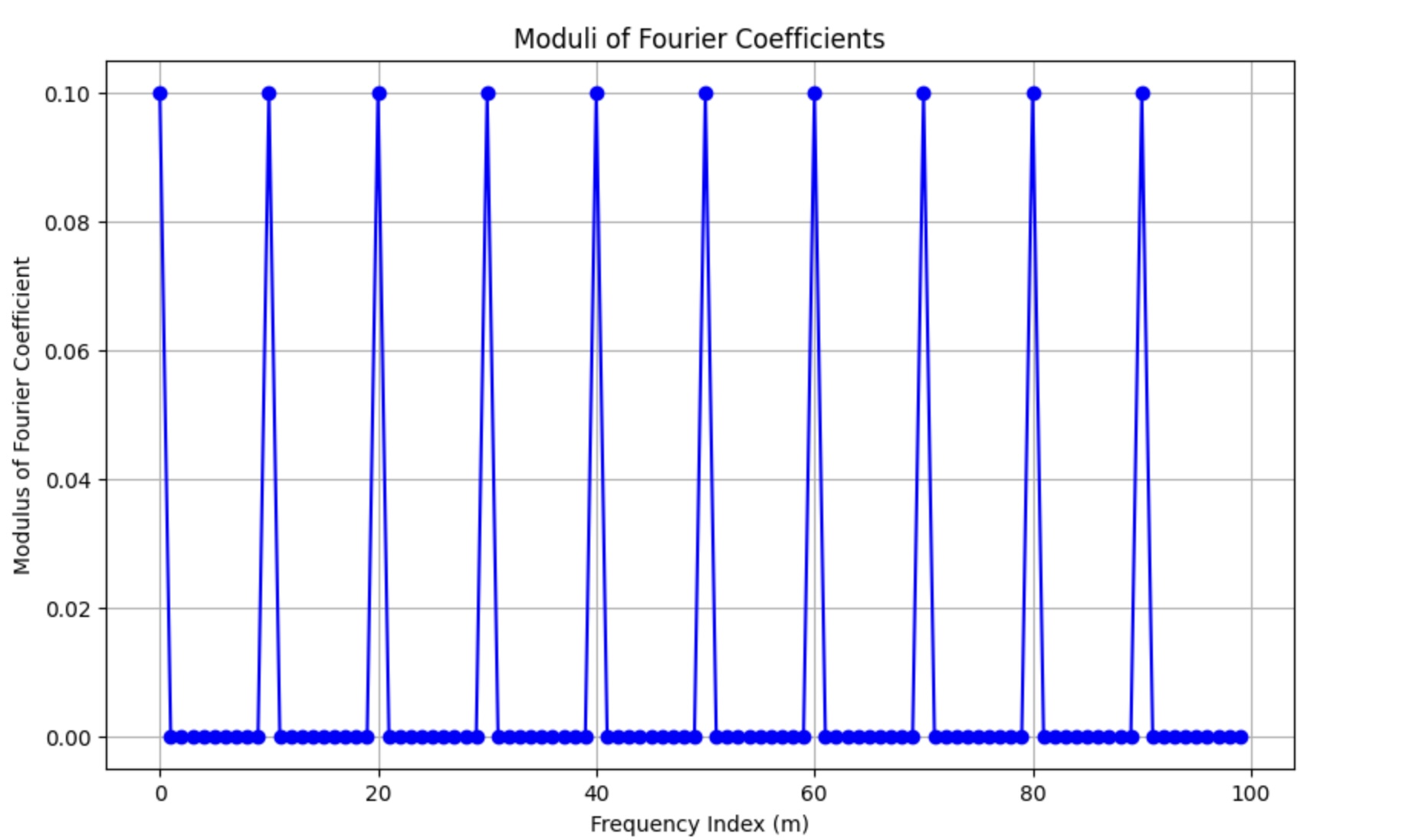}
 \caption{The discrete Fourier transform of an arithmetic progression.}
 \label{FTprogressionbinary}
\end{figure}

The key to proving part ii) Theorem \ref{energyuncertaintyprincipletheorem} is the following ``universal" restriction inequality. 

\begin{theorem}[Theorem 3.12, \cite{IM24}] \label{energyrestrictiontheorem} Let $f: {\mathbb Z}_N^d \to {\mathbb C}$ and let $\Sigma$ be a subset of ${\mathbb Z}_N^d$. Then
\begin{equation} \label{explicitconstantsrestrictionequation} {\left( \frac{1}{|\Sigma|} \sum_{m \in \Sigma} {|\widehat{f}(m)|}^2 \right)}^{\frac{1}{2}} \leq {\left(\frac{|\Sigma|}{N^{\frac{d}{2}}} \right)}^{-\frac{1}{2}} \cdot {\left( \max_{U \subset \Sigma} \frac{\Lambda(U)}{{|U|}^2} \right)}^{\frac{1}{4}} \cdot N^{-d} \cdot {\left( \sum_{x \in {\mathbb Z}_N^d} {|f(x)|}^{\frac{4}{3}} \right)}^{\frac{3}{4}}. \end{equation} 
\end{theorem} 

Given part i) in Theorem \ref{energyuncertaintyprincipletheorem} with $\alpha=0$, we can improve Theorem \ref{DS} in many cases. It is important to note that the recovery condition in \eqref{DS}, like the classical uncertainty principle, is sharp. However, in all cases where the additive energy on either support is nontrivial, we observe a strictly lower bound. Ultimately, this yields the following improved recovery condition.

\begin{corollary} Let $f:\mathbb{Z}_N^d\rightarrow\mathbb{C}$ supported in $E\subset\mathbb{Z}_N^d$. Suppose $\{\hat{f}(m)\}_{m\in S}$ are unobserved for some $S\subset\mathbb{Z}_N^d$. Then the signal $f$ can be recovered uniquely if 
\[ \min \left\{ \Lambda^\frac{1}{3}(S)|E|, \ \max_{U \subset S} \frac{\Lambda(U)}{{|U|}^2}|E|  \right\}<\frac{N^d}{2}.\]

\end{corollary}

The proof of this corollary follows by the same argument as in the proof of Theorem \ref{DS}. More precisely, let $h=f-g$, where $f$ and $g$ are chosen, as in the argument in Theorem \ref{DS}. 

%Suppose that $f$ cannot be recovered uniquely. Then there exists a signal $g: {\mathbb Z}_N^d\rightarrow\mathbb{C}$ such that $|supp(f)|=|supp(g)|$, $\hat{f}=\hat{g}$ away from S, and $f \neq g$. Just as before, we define $h=f-g$. As before, $|supp(h)|\leq 2|supp(f)|$. Since $supp(\hat{h}) \subset S$, we can see that $\Lambda(supp(\hat{h})) \leq \Lambda(S)$. Here we apply \eqref{additiveenergyeq} and get $2|supp({h})|\Lambda^{\frac{1}{3}}(S) \geq N$. Therefore, if we assume that $|supp(h)|\Lambda^{\frac{1}{3}}(S) < \frac{N}{2}$, then we conclude that $h$ is identically 0 and $f=g$.

\begin{remark} A different approach to achieving stronger uncertainty is through $(p,q)$-restriction, a property of sets that has been thoroughly researched in the context of signal recovery by Iosevich and Mayeli in \cite{IM24}.
\end{remark} 

\subsection{Good energy and bad energy} An interesting situation arises when $f: {\mathbb Z}_N^d \to {\mathbb C}$ is a signal supported in $S$, with $S$ a disjoint union of $S_1$ and $S_2$, where the additive energy of $S_1$ is small, and the additive energy of $S_2$ is large.
The most obvious question is, if $S=S_1 \cup S_2$, where $S_1 \cap S_2 = \emptyset$, what can we say about $\Lambda(S_1 \cup S_2)$? This leads us to the following simple result. 

\begin{lemma} \label{lemma:weknowholder} Let $S \subset {\mathbb Z}_N^d$,  and let $S=\cup_{i=1}^n S_i$, with  $S_i$s pairwise disjoint. Then 
\begin{equation} \label{eq:sumofenergies} \Lambda(S) \leq n^3 \cdot \sum_{i=1}^n \Lambda(S_i). \end{equation} 
\end{lemma} 

\vskip.125in 

To prove the lemma, observe that 
\begin{align*} \Lambda(S)&=N^{3d} \sum_m {|\widehat{S}(m)|}^4=N^{3d} \sum_m {\left|\sum_{i=1}^n \widehat{S}_i(m)\right|}^4  \leq N^{3d} \sum_m {\left( \sum_{i=1}^n |\widehat{S}_i(m)| \right)}^4. 
\end{align*}

Now, by applying the  H\"older's inequality with  the conjugate pair $(4/3, 4)$, we can upper bound the  right-hand side of the inequality by 

\begin{align*}
 N^{3d} n^3 \sum_m \sum_{i=1}^n{|\widehat{S}_i(m)|}^4. 
\end{align*}
 Notice, this expression equals 
$$ n^3 \sum_{i=1}^n \sum_m {|\widehat{S}_i(m)|}^4 =n^3 \sum_{i=1}^n \Lambda(S_i).$$
This completes the proof of the lemma.   

\vskip.125in 

Taking Lemma \ref{lemma:weknowholder} into account, we revisit Figure \ref{FTrandombinary} and Figure \ref{FTprogressionbinary}. Suppose that $S \subset {\mathbb Z}_{200}$ is the union of the two sets considered in Figure \ref{FTrandombinary} and Figure \ref{FTprogressionbinary}. The plot of the moduli of the Fourier coefficients is shown in Figure \ref{FTmixedeven}. It is clear from the plot that in this case the large Fourier coefficients still dominate, which is consistent with the conclusions of Lemma \ref{lemma:weknowholder}. 

Finally, we consider the case that illustrates the Remark \ref{remark:dominate}. Here a subset of ${\mathbb Z}_{500}$ consisting of $20$ random points, and take a union with an arithmetic progression of length $5$. The plot of the moduli of Fourier coefficients is given by Figure \ref{FTmixeduneven}. It is visible that while there are still a few high peaks, they are dominiated by the random part of the set. 

\begin{figure}
\centering
\includegraphics[scale=.4]{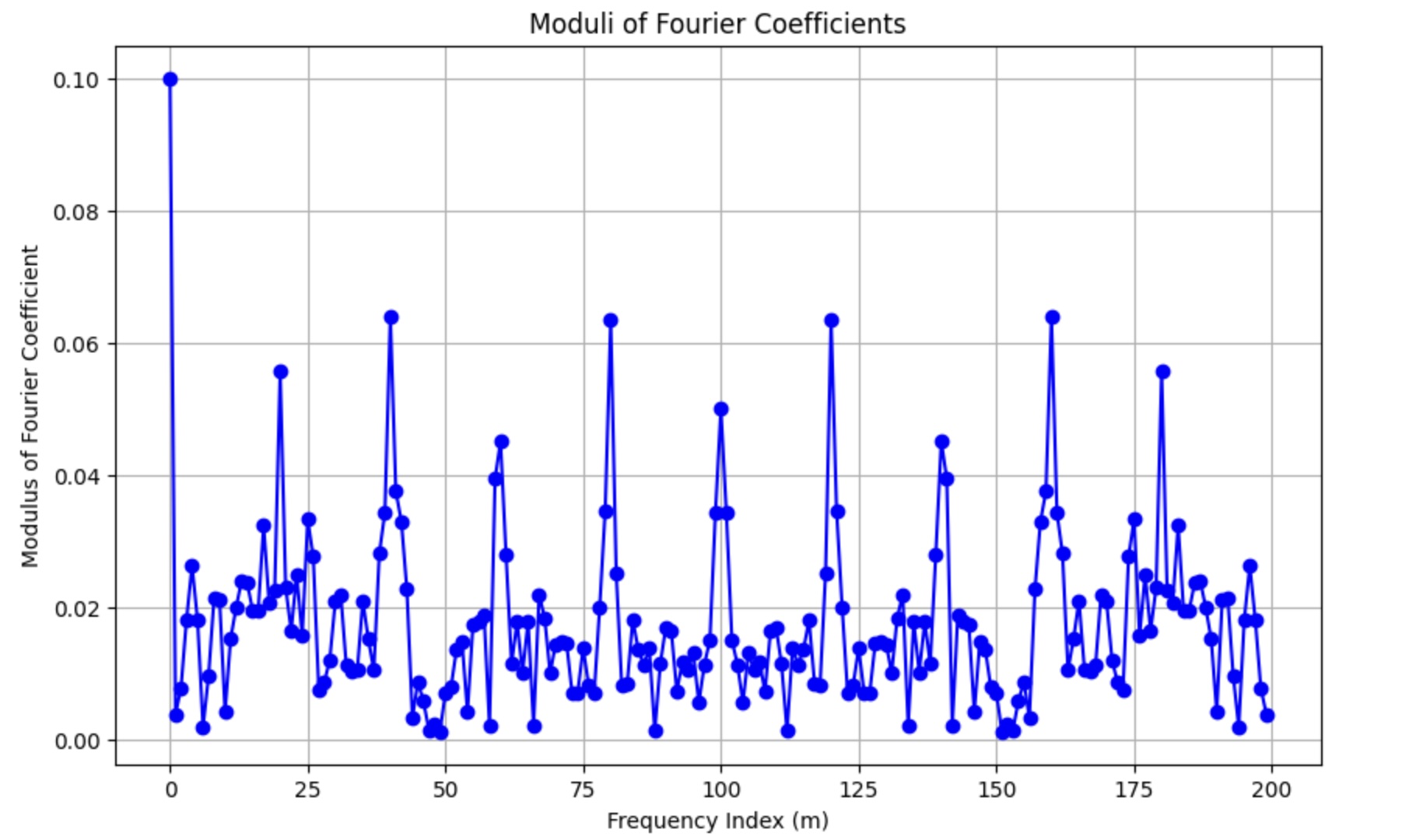}
 \caption{The discrete Fourier transform of the union of an arithmetic progression and a random set.}
 \label{FTmixedeven}
\end{figure}

\begin{figure}
\centering
\includegraphics[scale=.4]{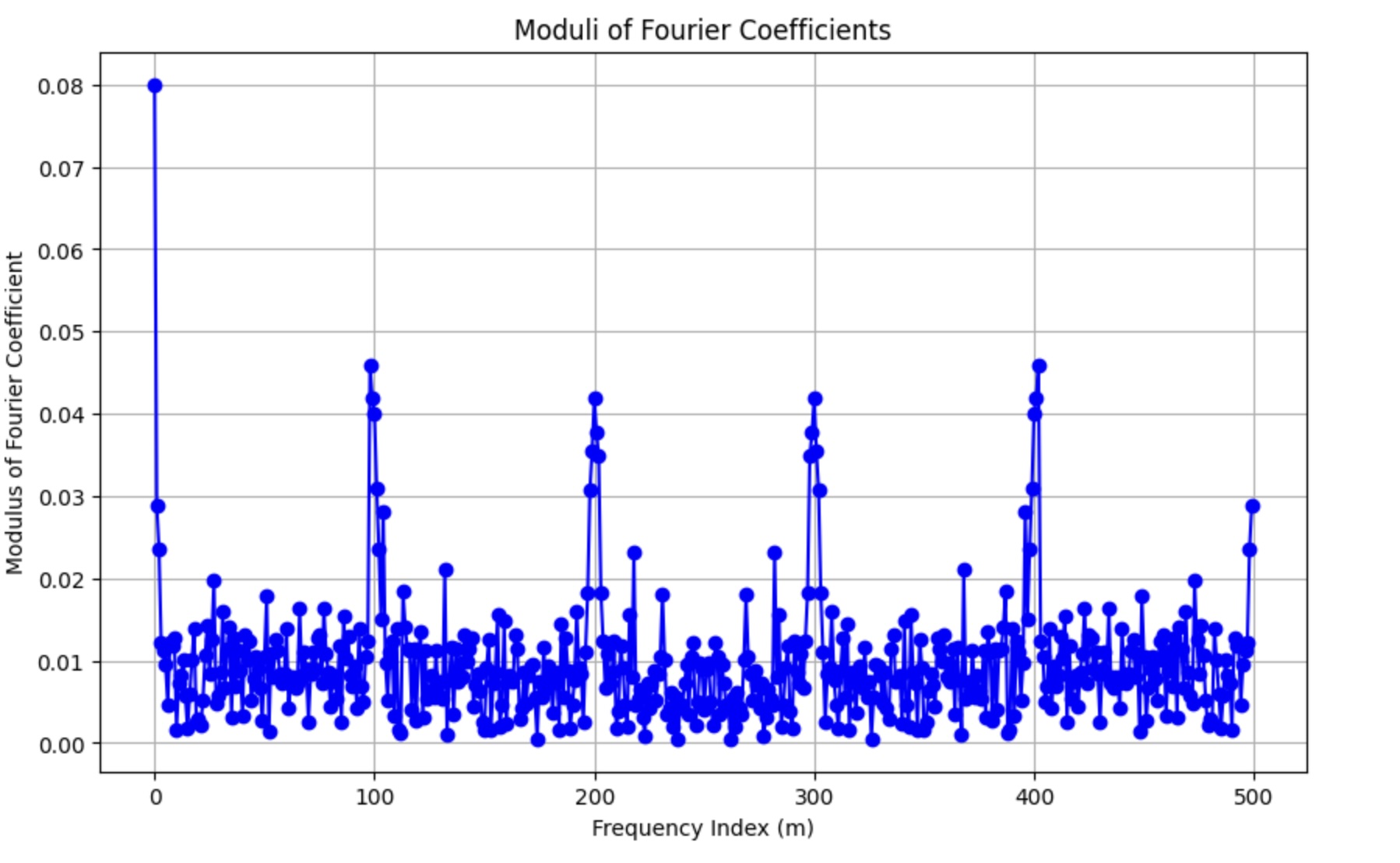}
 \caption{The discrete Fourier transform of a union of an arithmetic progression and a random set, with the random part of the set much larger in size.}
 \label{FTmixeduneven}
\end{figure}

\vskip.25in 

With Lemma \ref{lemma:weknowholder} in two, we obtain the following version of Theorem \ref{energyuncertaintyprincipletheorem}. 

\begin{theorem} \label{energyuncertaintyprincipletheoremenhanced} Let $f: {\mathbb Z}_N^d \to {\mathbb C}$ supported in $E \subset {\mathbb Z}_N^d$, with $\widehat{f}$ supported in $S \subset {\mathbb Z}_N^d$. Suppose that $S$ is a disjoint union of the sets $S_i$, $1 \leq i \leq n$. Then 
\begin{equation} \label{eq: sumsofenergiesup} |E| \cdot n^3 \cdot \sum_{i=1}^n \Lambda(S_i) \ge N^d. \end{equation}

Similarly, if $E$ is the disjoint union of the sets  $E_i$, $1 \leq i \leq m$, then 
\begin{equation} \label{eq: sumsofenergiesbackwardsup} |S| \cdot m^3 \cdot \sum_{i=1}^m \Lambda(E_i) \ge N^d. \end{equation} 
\end{theorem} 

\vskip.125in 

\begin{remark} \label{remark:dominate} To see how Theorem \ref{energyuncertaintyprincipletheoremenhanced} may be useful, suppose that $\Lambda(S_1) \leq 2{|S_1|}^2$. Observe that 
$$ \sum_{i=2}^n \Lambda(S_i) \leq \sum_{i=2}^n {|S_i|}^3 \leq (n-1) \cdot {\left( \max_{2 \leq i \leq n} |S_i| \right)}^3.$$ 

It follows that if 
$$ (n-1) \cdot {\left( \max_{2 \leq i \leq n} |S_i| \right)}^3<{|S_1|}^2, $$ say, then we see that the additive energy of $S$ is primarily dictated by the additive energy of $S_1$, regardless of how much additive structure the other  sets $S_i$, $2\leq i\leq n$ may possess. 
\end{remark}

\section{The $L^1$ and $L^2$ minimization signal recovery algorithms}
\label{section:minimization}

\vskip.125in 

The following idea can be found in \cite{DS89} based on the classical Logan phenomenon \cite{L65}. The formal proof was first given by Santosa and Symes \cite{SS86}. Suppose that $f: {\mathbb Z}_N^d \to {\mathbb C}$ is a signal supported in $E \subset {\mathbb Z}_N^d$ and the frequencies ${\{\widehat{f}(m)\}}_{m \in S}$ are missing. Consider the optimization problem 

\[
\begin{aligned}
& \min_{u:\mathbb{Z}_N^d\to\mathbb{C}} && \|u\|_{L^2({\mathbb Z}_N^d)} \\
& \text{subject to} && \widehat{u}(m)=\widehat{f}(m) \quad \text{for all } m \notin S, \\
& && |\mathrm{supp}(u)| = |E|.
\end{aligned}
\]
 Let $g$ be the minimizer of the $L^2$-norm   among all functions 
$u$ 
satisfying the given constraints: 

 \begin{equation} \label{equation:L2minimizer} g = \text{argmin}_u {||u||}_{L^2({\mathbb Z}_N^d)} \end{equation} 

Let $f=g+h$. Observe that if $h$ is supported in $T$, then $|T| \leq 2|E|$. Also observe that $\widehat{h}$ is supported in $S$. We have 
\begin{align*}
    {||h||}^2_{L^2(T)}&=\sum_{x \in T} {\left| \sum_{m \in S} \chi(x \cdot m) \widehat{h}(m) \right|}^2\ 
 & \leq |T| \cdot |S| \cdot N^{-d} \cdot {||h||}_{L^2(T)} \leq 2|E| \cdot |S| \cdot N^{-d} \cdot {||h||}_{L^2(T)}, \end{align*}
 where the second line follows from the first by using Cauchy-Schwarz followed by Plancherel. 

If $|E| \cdot |S|<\frac{N^d}{2}$, we have ${||h||}^2_{L^2(T)}<{||h||}^2_{L^2(T)}$. This contradiction implies that $h$ is identically $0$. 

\vskip.125in 

In the $L^1$ case, define 
\begin{equation} \label{equation:L1minimizer} g = arg min_u {||u||}_{L^1({\mathbb Z}_N^d)} \ \text{with the constraint} \ \widehat{g}(m)=\widehat{f}(m) \ \text{for} \ m \notin S. \end{equation} 

Let $f=g+h$. Then 
\begin{align}\notag{||g||}_{L^1({\mathbb Z}_N^d)} &= {||f-h||}_{L^1({\mathbb Z}_N^d)} \\\notag
& = {||f-h||}_{L^1(E)}+{||f-h||}_{L^1(E^c)} = {||f-h||}_{L^1(E)}+{||h||}_{L^1(E^c)} \\\label{equation:keyminimizationstepL1} 
& \ge {||f||}_{L^1(E)}-{||h||}_{L^1(E)}+{||h||}_{L^1(E^c)}. \end{align}  

It follows that if 
$$ {||h||}_{L^1(E)}<{||h||}_{L^1(E^c)},$$ then 
$$ {||g||}_{L^1({\mathbb Z}_N^d)} > {||f||}_{L^1({\mathbb Z}_N^d)},$$ which would lead to a contradiction since $g$ is the minimizer. It would follow that $h \equiv 0$ and $g \equiv f$. 

\vskip.125in 

To prove that ${||h||}_{L^1(E)}<{||h||}_{L^1(E^c)}$, observe that 
\begin{align*}|h(x)|&=\left| \sum_{m \in S} \chi(x \cdot m) \widehat{h}(m) \right|\\
&\leq N^{-d} \cdot |S| \cdot {||h||}_{L^1({\mathbb Z}_N^d)}=N^{-d} \cdot |S| \cdot \left({||h||}_{L^1(E)}+{||h||}_{L^1(E^c)} \right).
\end{align*}

Summing both sides over $E$, we obtain 
$$ {||h||}_{L^1(E)} \leq \delta \cdot \left({||h||}_{L^1(E)}+{||h||}_{L^1(E^c)} \right).$$

Subtracting and distributing terms, we obtain 
$$ {||h||}_{L^1(E)} \leq \frac{\delta}{1-\delta} {||h||}_{L^1(E^c)}.$$

As before, if $\delta<\frac{1}{2}$ we obtain ${||h||}_{L^1(E)}<{||h||}_{L^1(E^c)}$, which implies that $${||g||}_{L^1({\mathbb Z}_N)}>{||f||}_{L^1({\mathbb Z}_N)}.$$ This is a contradiction since $g$ is the $L^1$ minimizer. It follows that $h \equiv 0$ and we have shown that $f=g$. 

\vskip.125in 

Using Logan's method (see \ref{equation:L1minimizer} above), we present a $L^1$ minimization result that grants a stronger recovery condition when the missing frequency set has low additive energy. We also offer a $L^2$ minimization result which we derive using a different setup. 

\begin{theorem} \label{theorem:L1minimization} Let $f: {\mathbb Z}_N^d \to {\mathbb C}$ supported in $E \subset {\mathbb Z}_N^d$. Suppose that the frequencies ${\{\widehat{f}(m)\}}_{m \in S}$ are unobserved, $S \subset {\mathbb Z}_N^d$. Suppose that there exists $C>0$ and $\alpha \in [2,3]$ such that 
\begin{equation} \label{equation:regularenergy} \Lambda(S) \leq C{|S|}^{\alpha}. \end{equation} 

Let 
\begin{equation} \label{equation:L1threshold} \delta={\left( \frac{|E|\Lambda^{\frac{1}{3}}(S)}{N^d} \right)}^{\frac{3}{4}}. \end{equation} 

Then if 
$$ \delta<\frac{1}{2},$$ and $g$ is the $L^1$ minimizer as in \eqref{equation:L1minimizer}, then $f \equiv g$. 

\end{theorem}  

\begin{theorem} \label{theorem:L2minimization} Let $f: {\mathbb Z}_N^d \to {\mathbb C}$ supported in $E \subset {\mathbb Z}_N^d$. Suppose that the frequencies ${\{\widehat{f}(m)\}}_{m \in S}$ are unobserved, $S \subset {\mathbb Z}_N^d$. Suppose that there exists $C>0$ and $\alpha \in [2,3]$ such that 
\begin{equation} \label{equation:nestedenergy} \Lambda(U) \leq C{|U|}^{\alpha} \ \text{for all} \ U \subset S. \end{equation} 

For any $\beta>\alpha$, let 
\begin{equation} \label{equation:L2threshold} \delta_{\beta}={\left( \frac{|E| \cdot {|S|}^{\beta-2}}{N^d} \right)}^{\frac{1}{4}}. \end{equation}

Then if 
$$ \delta_{\beta}<\frac{1}{2\left(17+\frac{1}{1-2^{-(\beta-\alpha)}} \right) \cdot {(2C)}^{\frac{1}{4}}},$$ and $g$ is the $L^2$ minimizer as in \eqref{equation:L2minimizer}, then $f \equiv g$. \end{theorem} 

\vskip.25in 

\section{Proof of Theorem \ref{energyuncertaintyprincipletheorem}}
\label{proofsection}

\vskip.25in 
\begin{proof}[Proof of Theorem \ref{energyuncertaintyprincipletheorem}] We first prove part i) Let $f:\mathbb{Z}_N^d \rightarrow \mathbb{C}$ with support in $E$ and $supp(\widehat{f})=\Sigma$. So 
\[|f(x)|=\left|\sum_{m\in\Sigma}\chi(x\cdot m)\widehat{f}(m)\right|
\leq| \Sigma|^{\frac{3}{4}} {\left(\sum_{m\in\mathbb{Z}_N^d}|\widehat{f}(m)|^4 \right)}^{1/4}\cdot\]
With Fourier Inversion, we get
\[=|\Sigma|^{\frac{3}{4}}N^{-d} {\left(\sum_{m\in\mathbb{Z}_N^d} {\left|\sum_{x\in E}\chi(-x\cdot m)f(x)\right|}^4 \right)}^{\frac{1}{4}}.\]
By expanding in $x_{1}, x_{2}, x_{3}, x_{4}$ 
\begin{align*}
&=|\Sigma|^{\frac{3}{4}}N^{-d} {\left(\sum_{x_1,x_2,x_3,x_4\in E}\sum_{m\in\mathbb{Z}_N^d}\chi((x_3+x_4-x_1-x_2)\cdot m)f(x_1)f(x_2)\overline{f(x_3)f(x_4)}\right)}^{\frac{1}{4}} \\ 
&=|\Sigma|^{\frac{3}{4}} N^{{-\frac{3d}{4}}} {\left(\sum_{x_1+x_2=x_3+x_4}f(x_1)f(x_2)\overline{f(x_3)f(x_4)}\right)}^{\frac{1}{4}}\\&\leq|\Sigma|^{\frac{3}{4}}N^{\frac{3d}{4}} \Lambda^{\frac{1}{4}}(E)||f||_{\infty}.
\end{align*}

Since this bound on  $f(x)$ is true for all x, it is also true for the maximum: 
\[||f||_{\infty}\leq|\Sigma|^{\frac{3}{4}}N^{{-\frac{3d}{4}}}\Lambda^{\frac{1}{4}}(E)||f||_{\infty}.\]  
It follows that if $f$ is not identically $0$, the preceding bound implies that  
$$N^d\leq|\Sigma|\Lambda^{\frac{1}{3}}(E).$$
Starting with $|\widehat{f}(m)|$ instead of $|f(x)|$, we get a symmetric result, which we scale to yield our theorem. This completes the proof of part i). 
    
\vskip.125in 
    
We now prove Part ii). By (\ref{explicitconstantsrestrictionequation}), 
$${\left( \frac{1}{|\Sigma|} \sum_{m \in \Sigma} {|\widehat{f}(m)|}^2 \right)}^{\frac{1}{2}} \leq {\left(\frac{|\Sigma|}{N^{\frac{d}{2}}} \right)}^{-\frac{1}{2}} \cdot {\left( \max_{U \subset \Sigma} \frac{\Lambda(U)}{{|U|}^2} \right)}^{\frac{1}{4}} \cdot N^{-d} \cdot {\left( \sum_{x \in E} {|f(x)|}^{\frac{4}{3}} \right)}^{\frac{3}{4}}.$$    

It follows that 

$$ {\left( \sum_{m \in \Sigma} {|\widehat{f}(m)|}^2 \right)}^{\frac{1}{2}} \leq {|\Sigma|}^{\frac{1}{2}} \cdot {\left(\frac{|\Sigma|}{N^{\frac{d}{2}}} \right)}^{-\frac{1}{2}} \cdot {\left( \max_{U \subset \Sigma} \frac{\Lambda(U)}{{|U|}^2} \right)}^{\frac{1}{4}} \cdot N^{-d} \cdot {\left( \sum_{x \in E} {|f(x)|}^{\frac{4}{3}} \right)}^{\frac{3}{4}}.$$
Since $\widehat{f}$ is supported in $\Sigma$, we can apply the Plancherel identity to  obtain 
$$ N^{-\frac{d}{2}} {\left( \sum_{x \in E} {|f(x)|}^2 \right)}^{\frac{1}{2}} \leq {|\Sigma|}^{\frac{1}{2}} \cdot {\left(\frac{|\Sigma|}{N^{\frac{d}{2}}} \right)}^{-\frac{1}{2}} \cdot {\left( \max_{U \subset \Sigma} \frac{\Lambda(U)}{{|U|}^2} \right)}^{\frac{1}{4}} \cdot N^{-d} \cdot {\left( \sum_{x \in E} {|f(x)|}^{\frac{4}{3}} \right)}^{\frac{3}{4}}.$$

Applying H\"older's inequality, we obtain 
$$ N^{-\frac{d}{2}} {\left( \sum_{x \in E} {|f(x)|}^2 \right)}^{\frac{1}{2}} \leq {|\Sigma|}^{\frac{1}{2}} \cdot {\left(\frac{|\Sigma|}{N^{\frac{d}{2}}} \right)}^{-\frac{1}{2}} \cdot {\left( \max_{U \subset \Sigma} \frac{\Lambda(U)}{{|U|}^2} \right)}^{\frac{1}{4}} \cdot {|E|}^{\frac{1}{4}} \cdot N^{-d} \cdot {\left( \sum_{x \in E} {|f(x)|}^2 \right)}^{\frac{1}{2}}.$$

Since $f\neq 0$, the preceding inequality  leads to 
$$ N^{\frac{d}{4}} \leq {\left( \max_{U \subset \Sigma} \frac{\Lambda(U)}{{|U|}^2} \right)}^{\frac{1}{4}} \cdot {|E|}^{\frac{1}{4}},$$ and it follows that 
$$ N^d \leq |E| \cdot \max_{U \subset \Sigma} \frac{\Lambda(U)}{{|U|}^2}.$$

Exactly the same argument with $f$ replaced by $\widehat{f}$ and $\Sigma$ replaced by $E$ yields 
$$ N^d \leq |S| \cdot  \max_{F \subset E} \frac{\Lambda(F)}{{|F|}^2}.$$

Combining the two estimates yields the conclusion of Part ii), and the proof  of Theorem \ref{energyuncertaintyprincipletheorem} is complete.
\end{proof}
 
\section{Proof of Theorem \ref{theorem:L1minimization}}
\label{section:L1minimizationproof}

Following the classical argument presented before the proof of Theorem \ref{theorem:L1minimization}, namely (\ref{equation:keyminimizationstepL1}), we see that 
\begin{equation} \label{equation:setupL1} {||g||}_{L^1({\mathbb Z}_N^d)} \ge {||f||}_{L^1({\mathbb Z}_N^d)}- {||h||}_{L^1(E)} + {||h||}_{L^1(E^c)}, \end{equation} where $g$ is as in (\ref{equation:L1minimizer}) and $h=f-g$. Observe that $\widehat{h}$ is supported in $S$. We must show that 
\begin{equation} \label{equation:keycontradictioninequalityL1} {||h||}_{L^1(E)}<{||h||}_{L^1(E^c)}. \end{equation} 

This will result in a contradiction as it will imply that 
$$ {||g||}_1>{||f||}_1,$$ which is impossible since $g$ is the $L^1$-minimizer. 

\vskip.125in 

To prove (\ref{equation:keycontradictioninequalityL1}), observe that by H\"older's inequality, 
 \begin{align*}{||h||}_{L^1(E)} &\leq 
 % {|E|}^{\frac{3}{4}} \cdot {||h||}_{L^4(E)} 
  {|E|}^{\frac{3}{4}} \cdot {||h||}_4\\&
  = {|E|}^{\frac{3}{4}} \cdot {\left( \sum_x \sum_{m,l,m',l'} \chi(x \cdot (m+l-m'-l')) \overline{\widehat{h}(m)} \overline{\widehat{h}(l)} \widehat{h}(m') \widehat{h}(l')  \right)}^{\frac{1}{4}}\\ 
& \leq {|E|}^{\frac{3}{4}} \cdot N^{-\frac{3d}{4}} \cdot \Lambda^{\frac{1}{4}}(S) \cdot {||h||}_1 =\delta_{\alpha}{||h||}_1. 
\end{align*}

It follows that 
$$ {||h||}_{L^1(E)} \leq \delta_{\alpha} \cdot \left( {||h||}_{L^1(E)}+{||h||}_{L^1(E^c)} \right), $$ which implies that 
$$ {||h||}_{L^1(E)} \leq \frac{\delta_{\alpha}}{1-\delta_{\alpha}} \cdot {||h||}_{L^1(E^c)},$$ and we conclude that 
$$ {||h||}_{L^1(E)}<{||h||}_{L^1(E^c)}$$ if $\delta_{\alpha}<\frac{1}{2}$. 

Plugging this into (\ref{equation:setupL1}) shows that ${||g||}_1>{||f||}_1$, which is a contradiction, since $g$ is the $L^1$ minimizer. It follows that $h \equiv 0$, so $f \equiv g$, as desired. 

\vskip.25in 

\section{Proof of Theorem \ref{theorem:L2minimization}}
\label{section:L2minimizationproof}

\vskip.125in 

As before, note that if $h$ is supported in $T$, then $|T| \leq 2|E|$, and $\widehat{h}$ is supported in $S$. We have 
\begin{align*} 
  {||h||}_{L^2(T)} &
  \leq
  {|T|}^{\frac{1}{4}} \cdot {||h||}_{L^4(T)} 
  \leq {|T|}^{\frac{1}{4}} \cdot {||h||}_{L^4({\mathbb Z}_N^d)}
  \\
  &
  ={|T|}^{\frac{1}{4}} \cdot N^{\frac{d}{4}} \cdot {\left( \sum_{m+l=m'+l'; m,l,m',l' \in S} \overline{\widehat{h}(m)} \overline{\widehat{h}(l)} \widehat{h}(m') \widehat{h}(l') \right)}^{\frac{1}{4}}. 
  \end{align*}

If $\widehat{h}(m)=U(m)$, $U \subset S$, then this expression is equal to 
\begin{align*}
{|T|}^{\frac{1}{4}} \cdot N^{\frac{d}{4}} \cdot \Lambda^{\frac{1}{4}}(U) &={|T|}^{\frac{1}{4}} \cdot N^{\frac{d}{4}} \cdot \Lambda^{\frac{1}{4}}(U){|U|}^{-\frac{1}{2}} \cdot {|U|}^{\frac{1}{2}}\\
 &= {|T|}^{\frac{1}{4}} \cdot N^{\frac{d}{4}} \cdot \Lambda^{\frac{1}{4}}(U){|U|}^{-\frac{1}{2}} \cdot {||\widehat{h}||}_{L^2({\mathbb Z}_N^d)} \\
 &={|T|}^{\frac{1}{4}} \cdot \Lambda^{\frac{1}{4}}(U){|U|}^{-\frac{1}{2}} \cdot N^{-\frac{d}{4}} \cdot {||h||}_{L^2(T)},
\end{align*}

and we conclude that 
 \begin{align}\notag {||h||}_{L^4(T)} &\leq \Lambda^{\frac{1}{4}}(U){|U|}^{-\frac{1}{2}} \cdot N^{-\frac{d}{4}} \cdot {||h||}_{L^2(T)}
\\\label{equation:keyenergysetup} 
& \leq C^{\frac{1}{4}} \cdot N^{-\frac{d}{4}} \cdot {|U|}^{\frac{\alpha}{4}-\frac{1}{2}} \cdot {||h||}_{L^2(T)} \end{align} 

when $\widehat{h}(m)=U(m)$, $U \subset S$. 
(Recall that $U(m)$ demotes the  indicator function of the set $U$.)
\vskip.125 in 

\begin{lemma} \label{lemma:restrictedstrongtype} The inequality \eqref{equation:keyenergysetup} holds with an additional multiplicative factor of $C_{\alpha, \beta}=2(17+\frac{1}{1-2^{-(\beta-\alpha)}})$ for any $h$ such that $h$ is supported in $T$ and $\widehat{h}$ is supported in $S$.
\end{lemma} 

Assuming the lemma for the moment, we see that by the energy assumptions in Theorem \ref{theorem:L2minimization},
$$ {||h||}_{L^2(T)} \leq 2\left( 17+\frac{1}{1-2^{-(\beta-\alpha)}} \right) \cdot 2^{\frac{1}{4}} \cdot {|E|}^{\frac{1}{4}} \cdot C^{\frac{1}{4}} \cdot {|U|}^{\frac{\alpha-2}{4}} \cdot N^{-\frac{d}{4}} {||h||}_{L^2(T)}=3 \cdot {(2C)}^{\frac{1}{4}} \cdot \delta_{\alpha} \cdot {||h||}_{L^2(T)}.$$

By assumption, 
$$ 2\left( 17+\frac{1}{1-2^{-(\beta-\alpha)}} \right) \cdot {(2C)}^{\frac{1}{4}} \cdot \delta_{\alpha}<1,$$ so we obtain 
$$ {||h||}_{L^2(T)}<{||h||}_{L^2(T)},$$ which is a contradiction. We conclude that $h \equiv 0$, hence $f \equiv g$. 

\vskip.125in 

We are left to prove Lemma \ref{lemma:restrictedstrongtype}. This is a variant of the argument in \cite{KP15}, pages 11–12.
\vskip.125in
Note that since we may assume that $h$ is nonzero, we may scale $N^{-\frac{d}{2}}||h||_{L^{2}{(\mathbb{Z}_{N}^{d})}}$ to 1. This scaling additionally implies that 
$|\widehat{h}(m)|\leq 1$. We can prove this using the definition of the Fourier transform and the Cauchy-Schwartz inequality.
$$|\widehat{h}(m)|=|N^{-d}\sum_{x} \chi(-x\cdot m)h(x)|\leq N^{-\frac{d}{2}}\left(\sum_{x}|h(x)|^{2}\right)^\frac{1}{2}=1$$
\vskip.125in
Observe that we may assume that $\widehat{h}$ is real-valued. Define 
$$ U_j=\left\{m \in {\mathbb Z}_N^d: 2^{-j} \leq |\widehat{h}(m)| \leq 2^{-j+1} \right\}. $$

We need another simple observation. Under the assumptions of Theorem \ref{theorem:L2minimization}, 
\begin{equation} \label{equation:multienergy} \left|\{(m,l,m',l') \in U \times V \times U' \times V': m+l=m'+l' \}\right| \leq {(\Lambda(U)\Lambda(V)\Lambda(U')\Lambda(V'))}^{\frac{1}{4}}, \end{equation} where $U,V,U',V' \subset S$ and $C$ is the same as in (\ref{equation:nestedenergy}). To see this, let $f_U, f_V, f_{U'}, f_{V'}$ denote the inverse Fourier transforms of $U,V,U',V'$, respectively. Then the left-hand side of (\ref{equation:multienergy}) is equal to 
$$ N^{-d} \sum_{x \in {\mathbb Z}_N^d} |f_U(x)f_V(x)f_{U'}(x)f_{V'}(x)|.$$ 

By H\"older's inequality, this quantity is bounded by 
$$ N^{-d} {\left( \sum_x {|f_U(x)|}^4 \right)}^{\frac{1}{4}}{\left( \sum_x {|f_V(x)|}^4 \right)}^{\frac{1}{4}}{\left( \sum_x {|f_{U'}(x)|}^4 \right)}^{\frac{1}{4}}{\left( \sum_x {|f_{V'}(x)|}^4 \right)}^{\frac{1}{4}},$$ which equals 
$$ {\left( \Lambda(U)\Lambda(V)\Lambda(U')\Lambda(V')\right)}^{\frac{1}{4}},$$ as claimed in (\ref{equation:multienergy}). 

We are now ready to complete the proof of the lemma. We have 
 \begin{align}\notag
{||h||}_{L^4({\mathbb Z}_N^d)}&={\left( \sum_{x \in {\mathbb Z}_N^d} \sum_{m, l, m', l' \in S} \chi(x \cdot (m+l-m'-l')) \overline{\widehat{h}(m)} \overline{\widehat{h}(l)} \widehat{h}(m') \widehat{h}(l') \right)}^{\frac{1}{4}}\\\notag
& \leq N^{\frac{d}{4}} \cdot {\left( 2^4\sum_{j,k,j',k'} 2^{-(j+k+j'+k')} \Lambda^{\frac{1}{4}}(U_j)\Lambda^{\frac{1}{4}}(U_k)\Lambda^{\frac{1}{4}}(U_{j'})\Lambda^{\frac{1}{4}}(U_{k'}) \right)}^{\frac{1}{4}} 
\\\label{equation:prep}  & \leq 2N^{\frac{d}{4}} \cdot C^{\frac{1}{4}} \cdot \sum_j 2^{-j} {|U_j|}^{\frac{\alpha}{4}}. 
 \end{align}

Observe that 
$$ {||\widehat{h}||}_{L^2({\mathbb Z}_N^d)} \leq 2{\left( \sum_j 2^{-2j} |U_j| \right)}^{\frac{1}{2}}$$ by a direct calculation. It follows that 
$$ \sum_{j=M}^{\infty} 2^{-j} {|U_j|}^{\frac{\alpha}{4}} \leq {|S|}^{\frac{\alpha}{4}} \cdot 2^{-M+1}.$$ 

Consequently, 
\begin{equation} \label{equation:high} 2C^{\frac{1}{4}} \cdot N^{\frac{d}{4}} \cdot \sum_{j=M}^{\infty} 2^{-j} {|U_j|}^{\frac{\alpha}{4}} \leq 4C^{\frac{1}{4}} \cdot {\left( \frac{{|S|}^{\alpha-2}}{N^d} \right)}^{\frac{1}{4}} \cdot {||h||}_{L^2({\mathbb Z}_N^d)} \cdot  \left[ {|S|}^{\frac{1}{2}} \cdot 2^{-M} \right] \leq  2C^{\frac{1}{4}} \cdot {\left( \frac{{|S|}^{\alpha-2}}{N^d} \right)}^{\frac{1}{4}} \cdot {||h||}_{L^2({\mathbb Z}_N^d)} \end{equation} if we choose $M$ such that 
\begin{equation} \label{equation:Mvalue} 2^M \geq 2{|S|}^{\frac{1}{2}}\geq 2^{M-1}.\end{equation} 

On the other hand,
\begin{align} \label{equation:low} 2C^{\frac{1}{4}} \cdot N^{\frac{d}{4}} \cdot \sum_{j=0}^{M-1} 2^{-j} {|U_j|}^{\frac{\alpha}{4}-\frac{1}{2}} \cdot {|U_j|}^{\frac{1}{2}} &\leq 2C^{\frac{1}{4}} \cdot N^{\frac{d}{4}} \cdot {\left( \sum_{j=0}^{M-1} 2^{-2j} |U_j| \right)}^{\frac{1}{2}} \cdot {\left( \sum_{j=0}^{M-1} {|U_j|}^{\frac{\alpha-2}{2}} \right)}^{\frac{1}{2}}
 \\\notag 
& \approx 2C^{\frac{1}{4}} \cdot N^{-\frac{d}{4}} \cdot {||h||}_2 \cdot {\left( \sum_{j=0}^{M-1} {|U_j|}^{\frac{\alpha-2}{2}} \right)}^{\frac{1}{2}}.
\end{align}

We have 
$$ \sum_{j=0}^{M-1} {|U_j|}^{\frac{\alpha-2}{2}} =\sum_{\{j: |U_j|<2^{j \delta}\}} {|U_j|}^{\frac{\alpha-2}{2}} + 
\sum_{\{j: |U_j| \ge 2^{j \delta}\}} {|U_j|}^{\frac{\alpha-2}{2}}=I+II.$$

\vskip.125in 

By a direct calculation, using (\ref{equation:Mvalue}), we see that 
$$ |I| \leq 2^{M \delta \frac{\alpha-2}{2}} \leq {\left( 4 {|S|}^{\frac{1}{2}} \right)}^{\delta \frac{\alpha-2}{2}} \leq {(16|S|)}^{\delta \frac{\alpha-2}{4}},$$ while for any $\beta>\alpha$, 
$$ II \leq \sum_{\{j: |U_j| \ge 2^{j \delta}\}} {|U_j|}^{\frac{\beta-2}{2}} 2^{-j \delta \frac{\beta-\alpha}{2}} 
\leq {|S|}^{\frac{\beta-2}{2}} \cdot \sum_{j=0}^{M-1} 2^{-j \delta \frac{\beta-\alpha}{2}} \leq {|S|}^{\frac{\beta-2}{2}} \cdot \frac{1}{1-2^{- \delta \frac{\beta-\alpha}{2}}}.$$

Taking $\delta=2$, we see that 
$$ 2C^{\frac{1}{4}} \cdot N^{\frac{d}{4}} \cdot \sum_{j=0}^{M-1} 2^{-j} {|U_j|}^{\frac{\alpha}{4}} 
\leq 2\left(16+\frac{1}{1-2^{-(\beta-\alpha)}} \right) \cdot C^{\frac{1}{4}} \cdot N^{-\frac{d}{4}} \cdot {||h||}_{L^2({\mathbb Z}_N^d)} \cdot {|S|}^{\frac{\beta-2}{4}}.$$

Putting everything together, we see that 
$$ {||h||}_{L^4(T)} \leq 2\left(17+\frac{1}{1-2^{-(\beta-\alpha)}} \right) \cdot C^{\frac{1}{4}} \cdot N^{-\frac{d}{4}} \cdot {|U|}^{\frac{\alpha}{4}-\frac{1}{2}} \cdot {||h||}_{L^2(T)}$$ and the proof is complete.

\end{document}